\documentclass{amsart}
\usepackage{amscd}
\usepackage{amssymb}
\theoremstyle{plain}
\newtheorem{theorem}{Theorem}[section]
\newtheorem{lemma}[theorem]{Lemma}
\newtheorem{proposition}[theorem]{Proposition}
\newtheorem{corollary}[theorem]{Corollary}
\newtheorem{algorithm}[theorem]{Algorithm}

\theoremstyle{definition}
\newtheorem{definition}[theorem]{Definition}
\newtheorem{def-prop}[theorem]{Definition-Proposition}

\theoremstyle{remark}
\newtheorem{remark}{Remark}[section]

\begin{document}
\title[Algorithms for quivers]{Some algorithms for semi-invariants of quivers.}
\address{117437, Ostrovitianova, 9-4-187, Moscow, Russia.}
\email{mitia@mccme.ru}
\author{D.A.Shmelkin}
\begin{abstract}
We present some theorems and algorithms for calculating
perpendicular categories and locally semi-simple decompositions.
We implemented a computer program {\sc TETIVA} based
on these algorithms and we offer this program for everybody's use.
\end{abstract}


\maketitle
\section{Introduction}

Let $Q$ be a finite quiver with the set $Q_0$ of vertices
and $Q_1$ of arrows; for an arrow $\varphi\in Q_1$ denote
by $t\varphi$ and $h\varphi$ its tail and its head, respectively.
For $\alpha\in {\bf Z}_+^{Q_0}$ denote
by $R(Q,\alpha)$ the set of the representations of $Q$ of dimension
$\alpha$ over an  algebraically closed field ${\bf k}$ of characteristic 0, i.e.,
$R(Q,\alpha) = \bigoplus_{\varphi\in Q_1} {\rm Hom}
({\bf k}^{\alpha_{t\varphi}},{\bf k}^{\alpha_{h\varphi}})$.
For $V\in R(Q,\alpha), a\in Q_0,\varphi\in Q_1$ denote by $V(a)$ the vector
space ${\bf k}^{\alpha_a}$ sitting at $a$ and by $V(\varphi):V(t\varphi)\to V(h\varphi)$
the component of $V$. For representations $U,V$ of $Q$ a homomorphism $f:U\to V$ is a collection 
$f=(f_a\vert a\in Q_0, f_a\in {\rm Hom}((U(a),V(a)))$ such that for each $\varphi\in Q_1$
holds $V(\varphi) f(t\varphi) = f(h\varphi)U(\varphi)$. By ${\rm Hom}(U,V)$ denote the vector space
of all homomorphisms.

Clearly, $R(Q,\alpha)$ is a ${\bf k}$-vector space and there is a natural
linear representation $\rho$ of the group $GL(\alpha) = \prod_{i\in Q_0} GL(\alpha_i)$ 
in $R(Q,\alpha)$: 
\begin{equation}
(\rho(g)(V))(\varphi) = g(h\varphi) V(\varphi) (g(t\varphi))^{-1},
\end{equation}
\noindent such that the orbits of $GL(\alpha)$ are the isomorphism classes of representations.

Recall that $\beta\in {\bf Z}_+^{Q_0}$ is called a {\it root} if $R(Q,\beta)$ contains
an indecomposable representation and a {\it Schur} root if generic element $U\in R(Q,\beta)$
is indecomposable, and in this case ${\rm Aut}(U)={\bf k}^*$ (\cite{kac}).
By the Krull-Schmidt theorem for any representation $V\in R(Q,\alpha)$ there is a 
decomposition $V=R_1\oplus R_2\oplus\cdots\oplus R_t$ into a sum of indecomposable summands and
it is unique up to permutations and isomorphisms of summands.
In particular, $V$ yields a decomposition $\alpha = \dim R_1 +\cdots + \dim R_t$ into the sum
of roots and this decomposition does not change over the isomorphism class of $V$.
This decomposition is an important invariant of $V$ and in some cases it allows to recover
the group ${\rm Aut}(V)$ or even the isomorphism class of $V$. The main subject of this
paper is to study several special classes of representations and to calculate the corresponding
decompositions.

It is shown in \cite{kac} that among all the decompositions of $\alpha$ into the sum of roots 
there is a generic element such that there is an open dense subset in $R(Q,\alpha)$ consisting 
of representations with such a decomposition; V.Kac called this {\it canonical} decomposition.
We however follow another tradition and prefer the name {\it generic} for the same object.
By \cite{kac} each root in the generic decomposition is a Schur root, and moreover
a decomposition $\alpha = \beta_1 +\cdots + \beta_t$ is generic if and only
if for generic $B_i\in R(Q,\beta_i)$ holds: ${\rm Aut}(B_i)={\bf k^*}$ and
${\rm Ext}(B_i,B_j)=0$ for $i,j=1,\cdots,t, i\neq j$. Recall that C.M.Ringel
introduced in \cite{ri} the {\it Euler} bilinear form 
\begin{equation}
\langle \alpha,\beta\rangle=\sum_{a\in Q_0} \alpha_a\beta_a-\sum_{\varphi\in Q_1} \alpha_{t\varphi}\beta_{h\varphi},
\end{equation}
\noindent such that the {\it Tits} form can be written as 
$q_Q(\alpha)=\langle \alpha,\alpha\rangle$, and proved the formula:
\begin{equation}\label{ringel}
\dim {\rm Ext}(U,V) = \dim {\rm Hom}(U,V) - \langle \dim U,\dim V\rangle.
\end{equation}

\noindent Applying the above formula to the generic decomposition $\alpha = \beta_1 +\cdots + \beta_t$,
we have $\langle \beta_i,\beta_j\rangle = \dim {\rm Hom}(B_i,B_j)\geq 0$, and in particular,
$\beta_i=\beta_j=\beta$ implies $q_Q(\beta)\geq 0$, so imaginary roots $\beta$ with $q_Q(\beta)< 0$
may occur in the generic decomposition only with multiplicity 1.
V.Kac found the above properties of the generic decomposition and addressed the problem
to find an algorithm for calculation of the generic decomposition in terms of the Euler form.
This is done in \cite{dw2}.

Denote by $SL(\alpha)$ the commutator subgroup of $GL(\alpha)$, 
$SL(\alpha)=\prod_{i\in Q_0} SL(\alpha_i)\subseteq GL(\alpha)$.
A natural task of the Invariant Theory in the quiver setup is to study
the $GL(\alpha)$-semi-invariant functions on $R(Q,\alpha)$, which are
also $SL(\alpha)$-invariant. To be precise, the character group of $GL(\alpha)$ 
is isomorphic to ${\bf Z}^{Q_0}$
such that $\chi\in {\bf Z}^{Q_0}$ gives rise to the character
$\overline{\chi}=\prod_{a\in Q_0,\alpha_a>0}\det_a^{\chi_a}$ and we have:
\begin{equation}
{\bf k}[R(Q,\alpha)]^{SL(\alpha)} = 
\oplus_{\chi\in {\bf Z}^{Q_0}} {\bf k}[R(Q,\alpha)]^{(GL(\alpha))}_{\chi},
\end{equation}
where ${\bf k}[R(Q,\alpha)]^{(GL(\alpha))}_{\chi} = 
\{f\in {\bf k}[R(Q,\alpha)]\vert 
gf = \overline{\chi} (g)f, \forall g\in GL(\alpha)\}$.
Note also that ${\bf k}[R(Q,\alpha)]^{(GL(\alpha))}_{\chi}\neq\{0\}$
implies $\chi(\alpha)\doteq\sum_{a\in Q_0}\chi_a\alpha_a = 0$.

A.Schofield introduced in \cite{sch91} a correspondance between representations
and semi-invariants. Namely, for any representation $W$ there is a determinantal semi-invariant
$c_W$ such that $c_W(V)\neq 0$ if and only if $\langle\dim V,\dim W\rangle = 0$ and
${\rm Hom}(V,W)$ $=0$, hence also ${\rm Ext}(V,W)=0$ by Ringel formula (\ref{ringel}). 
Moreover, the weight of $c_W$ is equal $-\langle\cdot,\dim W\rangle$.
Besides, the representations $V$ such that $c_W(V)\neq 0$ constitute an abelian subcategory
closed under homomorphisms, extensions, direct sum and summands and this subcategory
is denoted by $^{\perp}W$. Similarly, the subcategory $W^{\perp}$ consists of those $V$
such that ${\rm Hom}(W,V)$ $=0$ and ${\rm Ext}(W,V)=0$.

Assume that $Q$ has no oriented cycles. Then for any $\alpha$ and $\chi$ 
the vector space ${\bf k}[R(Q,\alpha)]^{(GL(\alpha))}_{\chi}$ is finite dimensional
and it is proved in \cite{dw0} that this vector space is generated
by semi-invariants $c_W$, where $W$ is a representation such that $-\langle\cdot,\dim W\rangle=\chi$.

In \cite{sh} we introduced a class of representations, which help to study the semi-invariants
of quiver from the geometrical point of view:

\begin{theorem}(\cite{sh})\label{char-lss} Let $V=m_1S_1+\cdots+m_tS_t\in R(Q,\alpha)$ be a decomposition
into indecomposable summands. The following properties of $V$ are equivalent:

(i) the $SL(\alpha)$-orbit of $V$ is closed in $R(Q,\alpha)$

(ii) the $GL(\alpha)$-orbit of $V$ is closed in $R(Q,\alpha)_{f}$, 
$f\in {\bf k}[R(Q,\alpha)]^{(GL(\alpha))}_{\chi}$

(iii) $S_1,\cdots,S_t$ are simple objects in $^{\perp}W$ 
for a representation $W$. 
\end{theorem}

We called the representations meeting the equivalent properties of the above Theorem 
{\it locally semi-simple}. In particular, these representations meet the formula
\begin{equation}\label{cronk}
\dim {\rm Hom}(S_i,S_j) = \delta_{ij}.
\end{equation}

\noindent This property yields an equality ${\rm Aut}(V)\cong GL(m_1)\times\cdots\times GL(m_t)$
so the decomposition corresponding to $V$ completely defines the embedding ${\rm Aut}(V)\subseteq GL(\alpha)$.
Recall that D.Luna introduced in \cite{lu} a stratification of the quotient $L/\!\!/G$
of a finite dimensional module $L$ over a reductive group $G$,
$L/\!\!/G= \sqcup (L/\!\!/G)_{(M)}$, where $(L/\!\!/G)_{(G)}$ is the subset of those
$\xi$ such that the unique closed orbit over $\xi$ is $G$-isomorphic to $G/M$. 
We introduced in \cite{sh} a sort of specification of the Luna statification of $L/\!\!/(G,G)$ 
by the strata $(L/\!\!/(G,G))^G_{(M)}$ such that the $G$-orbit of the points
on the unique closed orbit in the fibers are $G$-isomorphic to $G/M$. Each
usual Luna stratum of $L/\!\!/(G,G)$ is therefore decomposed into finitely many
locally closed substrata. By Theorem \ref{char-lss}, in the case of 
$G=GL(\alpha),L=R(Q,\alpha)$ this $GL(\alpha)$-stratification of $R(Q,\alpha)/\!\!/SL(\alpha)$ 
is equivalent to the
description of all locally semi-simple decompositions of $\alpha$.
Of particular interest are the open stratum and the corresponding decomposition
that we called {\it generic locally semi-simple}.

The results of the paper are as follows. Firstly, using
the results of \cite{dw6} we get in Theorem \ref{th1} a sufficient condition
for a decomposition to be locally semi-simple. 
Next, we consider an important particular case of a {\it prehomogeneous}
dimension vector $\beta$, i.e., such that
$R(Q,\beta)$ contains a dense orbit $GL(\beta)W$.
Actually, we revisit an important theorem in \cite{sch91} saying that the category $W^{\perp}$ 
is isomorphic to the category of representations
of a quiver $\Sigma$ without oriented cycles. We analyze the proof of that theorem
and find that it yields Algorithm \ref{alg_perp} for calculating the dimensions
of the simple objects in $W^{\perp}$. This Algorithm can be viewed,
firstly, as a tool for calculating the algebraically independent
generators of ${\bf k}[R(Q,\beta)]^{SL(\beta)}$ (see Theorem \ref{generator}).
Secondly, with the help of the Algorithm we get in Theorem \ref{prehstart}
a complete description of the Luna $GL(\beta)$-stratification,
and in particular, the generic locally semi-simple decomposition of $\beta$ in Corollary \ref{preh_lss}. 
Finally, we provide an Algorithm \ref{alg_lss} for calculating  the generic
locally semi-simple decomposition for arbitrary dimension vector $\alpha$. 
This Algorithm is based on one hand, on the idea
of that for the generic decomposition from \cite{dw0}, and on the other hand,
on Corollary \ref{preh_lss}.

Besides proving theorems and algorithms we implemented a computer program
for doing all these types of calculations, namely, allowing to calculate
generic and generic locally semi-simple decompositions for arbitrary dimension
vectors and perpendicular categories for a prehomogeneous vector. This program
is called {\sc TETIVA} and is available at \cite{te}

\section{Schur sequences and locally semi-simple representations.}

In this section we relate locally semi-simple decompositions of dimension vector
to various other decompositions and start with the $\sigma$-stable ones. 

Let $\alpha$ be a dimension vector and $\sigma\in {\bf Z}^{Q_0}$
be a weight such that $\sigma(\alpha)=0$. Recall that King
introduced in \cite{ki} the notion of (semi-)stability
of representations of dimension $\alpha$. Assume 
that generic representation of dimension $\alpha$ is $\sigma$-semi-stable,
or, equivalently, ${\bf k}[R(Q,\alpha)]_{\sigma}^{(GL(\alpha))}\neq \{0\}$.
Then each $\sigma$-semi-stable $V\in R(Q,\alpha)$ has a filtration 
in the subcategory of $\sigma$-semi-stable
representations with the $\sigma$-stable factors, that is,
Jordan-H\"older factors. So $V$ yields a decomposition of $\alpha$ 
into the linear combination of the dimensions of $\sigma$-stable representations,
and for $V$ generic we get the so-called {\it $\sigma$-stable}
decomposition of $\alpha$:
\begin{equation}
\alpha = m_1\alpha_1+\cdots+m_s\alpha_s.
\end{equation}
\noindent
Note that for $Q$ being a tame quiver and $\sigma$ the {\it defect}
the $\sigma$-stable representations are the regular ones
and the $\sigma$-stable decomposition is Ringel's {\it canonical} one,
see \cite{ri}. By \cite[Proposition ~10,~Theorem~11]{sh}, we get:

\begin{proposition}\label{sigma_SS}
The $\sigma$-stable decomposition is locally semi-simple.
\end{proposition}

\begin{definition}
For dimension vectors $\alpha,\beta$ denote by
$hom(\alpha,\beta)$ and $ext(\alpha,\beta)$
the dimensions of ${\rm Hom}(A,B)$ and ${\rm Ext}(A,B)$
for generic $A\in R(Q,\alpha),B\in R(Q,\beta)$, respectively.
Further, write $\alpha\perp \beta$ if $hom(\alpha,\beta)=0=ext(\alpha,\beta)$.
\end{definition}

\begin{definition}
A sequence $\alpha_1,\cdots,\alpha_s$ of dimension vectors is called {\it (right) perpendicular}
if it consists of Schur roots and for $1\leq i < j\leq s$ holds $\alpha_i\perp \alpha_j$.
\end{definition} 

Now we introduce an important notion from \cite{dw6}:

\begin{definition}
A perpendicular sequence $\alpha_1,\cdots,\alpha_s$ is called a {\it Schur} sequence if
\begin{equation}\label{pperp}
\alpha_i \circ \alpha_j \doteq \dim {\bf k}[R(Q,\alpha_i)]^{(GL(\alpha_i))}_{-\langle \cdot,\alpha_j\rangle} = 1,
1\leq i < j \leq s.
\end{equation}
\noindent This sequence is called
{\it quiver Schur} if additionally $\langle \alpha_j,\alpha_i\rangle\leq 0$
for $i < j$.
\end{definition}

The notion of quiver Schur sequence can be interpreted in terms of local quiver.
The idea of the latter goes back to \cite{lbp}, where it was applied for semi-simple
representations and we used it in \cite{sh} for locally semi-simple ones, too.
Moreover, the definition works for any representation $V=m_1S_1+\cdots+m_tS_t$ that meets 
condition (\ref{cronk}) and we define $\Sigma_V$ to be the quiver with $t$ vertices corresponding to the summands 
$S_1,\cdots,S_t$ and $\delta_{ij}-\langle\dim S_i,\dim S_j\rangle$ arrows from
$i$ to $j$. Note that, thanks to the condition (\ref{cronk})
and the Ringel formula (\ref{ringel}), $\delta_{ij}-\langle\dim S_i,\dim S_j\rangle =\dim {\rm Ext}(S_i,S_j)\geq 0$.
On the other hand, the definition of $\Sigma_V$ only depends on the sequence
$\underline{\alpha} = (\dim S_1,\cdots,\dim S_t)$, not on the multiplicities,
and even not on the summands themselves, provided the homomorphism
spaces are trivial. So it is possible and even more correct to denote
the quiver $\Sigma_{\underline{\alpha}}$.
This local quiver plays a crucial role in \cite{lbp} and \cite{sh} because 
for locally semi-simple $V$ the {\it slice representation} at $V$ is (by formula (9) in \cite{sh}):
\begin{equation}\label{slice}
({\rm Aut}(V),{\rm Ext}(V,V))\cong (GL(\gamma),R(\Sigma_{\underline{\alpha}},\gamma)),
\end{equation}
where $\gamma=(m_1,\cdots,m_t)$ is a dimension vector for $\Sigma_{\underline{\alpha}}$.
Threfore Luna's \'etale slice Theorem relates the local equivariant structure of a 
neighborhood of $V$ in $R(Q,\alpha)$ with that of $0$ in $R(\Sigma_{\underline{\alpha}},\gamma)$.

\begin{proposition}\label{prop2}
Let $\underline{\alpha}=(\alpha_1,\cdots,\alpha_t)$ be a sequence of Schur roots.

{\bf 1.} If $\underline{\alpha}$ is a quiver Schur sequence
then ${\rm hom}(\alpha_i,\alpha_j)=0$ for any $i\neq j$
and the quiver $\Sigma_{\underline{\alpha}}$ has no oriented cycles except loops.

{\bf 2.} If ${\rm hom}(\alpha_i,\alpha_j)=0$ for any $i\neq j$
and $\Sigma_{\underline{\alpha}}$ has no oriented cycles ex\-cept loops, then
after a reordering we have $\alpha_i\perp \alpha_j$ for $i < j$.
If moreover $\alpha_i$ is imaginary for at most one $i$,
then after a reordering $\underline{\alpha}$ becomes a quiver Schur sequence.
\end{proposition}

\begin{proof}
{\bf 1.} For $i < j$, $\alpha_i\perp\alpha_j$ implies
$hom(\alpha_i,\alpha_j)=0$ and $ext(\alpha_i,\alpha_j)=0$.
The latter together with \cite[Theorem~4.1]{sch92} implies 
that either $hom(\alpha_j,\alpha_i)=0$ or $ext(\alpha_j,\alpha_i)=0$.
By Ringel formula,  
$hom(\alpha_j,\alpha_i)-ext(\alpha_j,\alpha_i)=\langle \alpha_j,\alpha_i\rangle\leq 0$,
hence $hom(\alpha_j,\alpha_i)=0$ as well. Since all non-loop arrows
of $\Sigma_{\underline{\alpha}}$ go from $j$ to $i$ with $j > i$, 
this quiver does not contain non-loop oriented cycles.

{\bf 2.} Each oriented graph having no oriented cycles
admits an order such that all arrows go from bigger to smaller indices;
forget the loops of $\Sigma_{\underline{\alpha}}$ and fix such an order. 
Since ${\rm hom}(\alpha_i,\alpha_j)=0$ for any $1\leq i\neq j\leq t$
and ${\rm ext}(\alpha_i,\alpha_j)=0$ for any $i < j$ in our order, we have
$\alpha_i\perp\alpha_j$. If moreover, at least one from $\alpha_i,\alpha_j$
is real then $\alpha_i\circ\alpha_j = 1$ by \cite[Lemma~4.2]{dw6}. 
\end{proof}

Therefore the quiver Schur sequences are very close to the sequences
with trivial mutual homorphisms and without oriented cycles.
Of course, not each locally semi-simple representation meets
the latter condition: for instance take a tame quiver as $Q$
and the sum of the simple non-homogeneous regular representations
over an orbit of Coxeter functor; then this representation is locally
semi-simple by \cite[Proposition~20]{sh} but the local quiver
is a oriented cycle by \cite[Proposition~21]{sh}.

\begin{theorem}\label{th1}
For a quiver Schur sequence $\underline{\alpha} = (\alpha_1,\cdots,\alpha_t)$
and a tuple $(m_1,\cdots,$   $m_t)$ the decomposition
$\beta = m_1\alpha_1+\cdots+m_t\alpha_t$ is locally semi-simple.
\end{theorem}
\begin{proof}
By Theorem \ref{char-lss} the fact that the decomposition is locally semi-simple
does not depend on the multiplicities so we may assume $m_1=m_2=\cdots=m_t=1$.
Then we apply Theorem 5.1 from \cite{dw6}. Denote by $\Sigma(Q,\beta)$
the set of weights of the semi-invariants in ${\bf k}[R(Q,\alpha)]^{SL(\alpha)}$. 
Theorem 5.1 asserts that the cone ${\bf R}_+\Sigma(Q,\beta)$ has
a face $F={\bf R}_+\Sigma(Q,\beta)\cap \{\sigma\in {\bf R}^{Q_0}\vert \sigma(\alpha_1)=\cdots=\sigma(\alpha_t)=0\}$. 
Moreover, the Theorem guarantees that for $\sigma$ from the relative interior of $F$,
$\beta =\alpha_1+\cdots+\alpha_t$ is the $\sigma$-stable decomposition of $\beta$.
Now the assertion follows from Proposition \ref{sigma_SS}.
\end{proof}

\section{Prehomogeneous dimension vectors.}

Recall that a dimension vector $\beta$ is called {\it prehomogeneous}
if $R(Q,\beta)$ contains a dense $GL(\beta)$-orbit. By \cite{kac}
this is equivalent to the generic decomposition of $\beta$ containing
only real Schur roots.
For this particular case we are able to calculate the Luna stratification
completely.

\begin{proposition}\label{prop3}
If $\beta$ is prehomogeneous and $\underline{\alpha}$
is a sequence of Schur roots such that $\beta = m_1\alpha_1+\cdots+m_t\alpha_t$
then this decomposition is locally semi-simple
if and only if $\underline{\alpha}$ is a quiver Schur sequence up to order.
\end{proposition}

\begin{proof}
By Theorem \ref{th1} and Proposition \ref{prop2} we only need to prove that
if the decomposition is locally semi-simple then $\Sigma{\underline{\alpha}}$
does not contain oriented cycles (in particular, the absense of loops means that
all summands are real roots). Indeed, take $W=m_1S_1+\cdots+m_tS_t$
the locally semi-simple representation corresponding to the decomposition.
Then the orbit $GL(\beta)W$ is closed in an open affine neighborhood
$R_0\subseteq R(Q,\beta), R_0\ni W$ and by formula (\ref{slice}) the slice representation
at $W$  is isomorphic to $(GL(\gamma),R(\Sigma_{\underline{\alpha}},\gamma))$.
The \'etale slice Theorem yields an \'etale morphism of 
$R(\Sigma_{\underline{\alpha}},\gamma)/\!\!/GL(\gamma)$ to
$R_0/\!\!/GL(\beta)$. Since $GL(\beta)$ has an open orbit in $R(Q,\beta)$,
it is contained in $R_0$ and so $R_0/\!\!/GL(\beta)$ is a point.
Consequently, $R(\Sigma_{\underline{\alpha}},\gamma)/\!\!/GL(\gamma)$
is a point, hence, $\Sigma_{\underline{\alpha}}$ does not have oriented cycles.
\end{proof}

The way we compute the locally semi-simple decompositions of prehomogeneous
dimension vectors is based on Schofield's Theorem:

\begin{theorem} (\cite[Theorem~2.5]{sch91})\label{th1.5}
If $\beta$ is prehomogeneous and $W=m_1S_1+\cdots+m_tS_t$
is the decomposition into indecomposable summands of a reprsentation
from the dense orbit, then $W^{\perp}$ is isomorphic
to the category of representations of a quiver with $n-t$ vertices
and without oriented cycles, where $n = \vert Q_0\vert$. 
The same is true for $^{\perp}W$.
\end{theorem}

We want to generalize this Theorem and to give an algorithm for calculation
of the perpendicular category based on the original proof.
The above Theorem can be reformulated as follows: there are
$n-t$ representations $R_1,\cdots,R_{n-t}$, which are all non-isomorphic simple
objects in $W^{\perp}$. By Theorem \ref{char-lss} $R=R_1+\cdots+R_{n-t}$
is locally semi-simple and the above Theorem additionally asserts that
the local quiver $\Sigma$ of $R$ does not have oriented cycles.
A particular case of the Theorem is when $\beta$ is a real Schur root,
and actually the proof in \cite{sch91} deduces the general case from this particular one,
where the proof is based on the notion of projective and injective
representations that we recall following \cite{sch91}.

For $i,j\in Q_0$ denote by $[i,j]$ the ${\bf k}$-vector space on the basis 
of oriented paths from  $i$ to $j$ in $Q$.
For $a\in Q_0$ consider the representation $P_a$ such that $P_a(i)=[a,i],i\in Q_0$
and for any arrow $\varphi\in Q_1$ the map $P_a(\varphi)$ takes
a path $T$ from $a$ to $t\varphi$ to the concatenation $T\varphi$, which is a path 
$a$ to $h\varphi$.
Dually, consider the representation $I_a$ such that
$I_a(i)=[i,a]^*,i\in Q_0$ and $I_a(\varphi)$ is the dual map
to the natural one from $[h\varphi,a]$ to $[t\varphi,a]$.
These representations have nice properties with respect
to the homomorphisms and extensions: for any representation $V$ of $Q$ hold
\begin{equation}\label{projinj}
{\rm Hom}(P_a,V)\cong V(a), {\rm Hom}(V,I_a)\cong (V(a))^*,
{\rm Ext}(P_a,V)=0, {\rm Ext}(V,I_a)=0.
\end{equation}  

\begin{theorem}\label{th2}
Let $\beta$ be a real Schur root for a quiver $Q$ with $n$ vertices and
without oriented cycles and let $W$ be a generic representation of dimension $\beta$. 
Then:

{\bf 1.} If $\beta =\dim P_a$ for some $a\in Q_0$, then
the simple objects of the category $W^{\perp}$
are all the simple representations of $Q$ but $S_a$.
Otherwise, the dimensions of $n-1$ projective objects of $W^{\perp}$
are the indecomposable summands of the generic decomposition for dimension vectors
$\dim P_a -\langle \beta,\dim P_a\rangle\beta$, where $a$ runs over $Q_0$.

{\bf 2.} If $\beta =\dim I_a$ for some $a\in Q_0$, then
the simple objects of the category $^{\perp}W$
are all the simple representations of $Q$ but $S_a$.
Otherwise, the dimensions of $n-1$ injective objects of $^{\perp}W$
are the indecomposable summands of the generic decomposition for dimension vectors
$\dim I_a -\langle \dim I_a,\beta\rangle\beta$, where $a$ runs over $Q_0$.
\end{theorem}

\begin{proof}
We prove assertion {\bf 1}, the proof for {\bf 2} being similar.
First of all, for $\beta = \dim P_a$ the assertion follows
from formulae (\ref{projinj}). In the opposite case, 
we apply the proof of Theorem 2.3 and Theorem 3.1 from \cite{sch91}.
From the proof of Theorem 2.3 we learn that all indecomposable projective
objects in $W^{\perp}$ can be obtained as the indecomposable summands
of a generic extension of $sW$ by $\Lambda$, where 
$\Lambda=\sum_{a\in Q_0}P_a, s=\dim{\rm Ext}(W,\Lambda)$.
For each individual $P=P_a$ Theorem 3.1 considers a generic
exact sequence $0\to P\to P^{\sim} \to sW \to 0$, where
$s=\dim{\rm Ext}(W,P)$ and states that $P^{\sim}$ is projective
in $W^{\perp}$. Then mutual extensions of the indecomposable summands
of $P^{\sim}$ are trivial in $W^{\perp}$ by (\ref{projinj}), hence, trivial
in $Mod(Q)$, because $W^{\perp}$ is closed under extensions. Therefore,
$P^{\sim}$ is generic in its dimension and the dimensions
of the indecomposable summands of $P^{\sim}$ are
the summands of the generic decomposition of
$\dim P^{\sim}=\dim P +\dim{\rm Ext}(W,P)\beta=\dim P -\langle \beta,\dim P\rangle\beta$,
the latter equality following from ${\rm Hom}(W,P)=0$.
\end{proof}

The above Theorem yields a quick algorithm as follows:

\begin{algorithm}\label{alg_sch_perp} Right perpendicular category of a Schur root

{\sc input:} a quiver $Q$ with $n$ vertices and without oriented cycles,

\hspace{1.2cm}a real Schur root $\gamma$

{\sc output:} $n-1$ dimensions of the simple objects in $W^{\perp}$ such that
  
\hspace{1.5cm}$GL(\gamma)W$ is dense in $R(Q,\gamma)$.

1. Calculate the dimensions $\rho_1,\cdots,\rho_n$ of indecomposable projectives.

\hspace{0.5cm}If $\gamma=\rho_j$, then 
{\sc return} $\varepsilon_1,\cdots\varepsilon_{\hat{j}}\cdots,\varepsilon_n$.

2. Loop on $i=1,\cdots,n$

\hspace{1cm} Calculate the generic decomposition of 
             $\rho_i -\langle \gamma,\rho_i\rangle\gamma$

\hspace{1cm} Add each summand to the array provided it is not yet there

\hspace{0.3cm} After step 2 we must have distinct summands $\beta_1,\cdots,\beta_{n-1}$ in the array 

3. Color the entries $1,\cdots,n-1$ white, the number of black entries $b=0$

   Loop while $b < n-1$

\hspace{0.5cm} Loop on $j=1,\cdots,n-1$

\hspace{1cm} If $j$-th entry is white and for each other white entry $k$,
             $\langle \beta_k, \beta_j\rangle = 0$, then

\hspace{1cm} remember this entry $j$, which is going to be black, {\sc break} the loop

\hspace{0.5cm} Set next simple dimension $\alpha_j=\beta_j$

\hspace{0.5cm} Loop on $k=1,\cdots,n-1$

\hspace{1cm} If $k$-th entry is black then 
$\alpha_j = \alpha_j - \langle \beta_k, \beta_j\rangle\alpha_k$

\hspace{0.5cm} $b = b+1$.

4.{\sc return} $\alpha_1,\cdots,\alpha_{n-1}$.
\end{algorithm}

\begin{proof}
After Theorem \ref{th2} only the step 3 of the algorithm needs to be explained.
In that step we do the inverse to step 1, i.e., we recover the dimensions
of the simple objects from those of indecomposable projectives.
The idea of the step is that, though we do not know the quiver $\Sigma$
of the simple objects, we know that the Euler form on $\Sigma$
is the same that inherited from $Q$.
But by formulae (\ref{projinj})  $ext(\beta_i,\beta_j)=0$, hence
again by those formulae, $\langle \beta_i,\beta_j\rangle$ with respect to the Euler form
of $\Sigma$ is equal to the number of paths from $j$ to $i$ in $\Sigma$.
So the first vertex becomming black is a sink of $\Sigma$, hence,
the corresponding projective is simple. Furtermore,
each next vertex becoming black is the sink of $\Sigma$ with removed black vertices,
hence the corresponding projective is the simple plus the sum over black
vertices of the already obtained simple dimensions multiplied by the number of paths to there.
This completes the proof.
\end{proof}

\begin{remark}
Dualizing Algorithm \ref{alg_sch_perp} as in Theorem \ref{th2}, we get one
for the {\it left} perpendicular category.
\end{remark}

Now we generalize Theorem \ref{th1.5} as follows:

\begin{theorem}
If $W = m_1S_1+\cdots+m_tS_t$, where
$S_1,\cdots,S_t$ are Schur indecomposable summands and  $(\dim S_1,\cdots,\dim S_t)$ is a 
perpendicular sequence of real Schur roots, then there is a sequence
$\underline{\alpha}=(\alpha_1,\cdots,\alpha_{n-t}),n=\vert Q_0\vert$ 
of real Schur roots such that the corresponding indecomposable representations
are all simple objects in $W^{\perp}$ and $\Sigma_{\underline{\alpha}}$ has no oriented cycles.
The same is true for $^{\perp}W$.
\end{theorem}

\begin{proof}
We just present an algorithm for calculation of $\underline{\alpha}$ based
on Algorithm \ref{alg_sch_perp}:

\begin{algorithm}\label{alg_perp}\quad

{\sc input:} a quiver $Q$ with $n$ vertices and without oriented cycles,

\hspace{1.2cm}a perpendicular sequence $(\dim S_1,\cdots,\dim S_t)$
of real Schur roots

{\sc output:} dimensions $\alpha_1,\cdots,\alpha_{n-t}$ of the simple objects in $W^{\perp}$

  Loop on $i=1,\cdots,t$

\hspace{0.5cm} We have dimensions $\alpha_1,\cdots,\alpha_{n+1-i}$ of the simple 
objects of the current 

\hspace{0.5cm}
category. For $i=1$ the category is the whole of $Mod(Q)$.

\hspace{0.5cm} Calculate the quiver $\Sigma_i$ of the current category by means of Euler form

\hspace{0.5cm} Expand $\dim S_i$ as the linear combination of $\alpha_1,\cdots,\alpha_{n+1-i}$,

\hspace{0.5cm} Coefficients yield a dimension vector $\gamma_i$ for $\Sigma_i$

\hspace{0.5cm} Find the dimensions $\overline{\alpha_1},\cdots,\overline{\alpha_{n-i}}$
of the simple objects in the right 

\hspace{0.5cm} perpendicular category to $\gamma_i$ for $\Sigma_i$

\hspace{0.5cm} calculate new $\alpha_1,\cdots,\alpha_{n-i}$
as the linear combinations of old 

\hspace{0.5cm}
$\alpha_1,\cdots,\alpha_{n+1-i}$
with coefficients from $\overline{\alpha_1},\cdots,\overline{\alpha_{n-i}}$

{\sc return} $\alpha_1,\cdots,\alpha_{n-t}$.
\end{algorithm}

The key idea of this clear algorithm is that we obtain the perpendicular
category to $W$ as a sequence of subcategories, where the next is obtained
as the perpendicular category to a Schur root. The algorithm for the 
left perpendicular category is similar, we only go from $S_t$ to $S_1$.
\end{proof}

The above Theorem makes possible to introduce an operation $^{\perp}$
mapping a real (i.e., consisting of real roots) perpendicular sequence 
$\underline{\alpha}$ of $t$ roots to  real perpendicular sequences $\underline{\alpha}^{\perp}$
and $^{\perp}\underline{\alpha}$ of $n-t$ roots being the sequence
of dimensions of simple objects in the right and left perpendicular
categories, respectively. The first natural application of Algorithm \ref{alg_perp} 
is given by the following interpretation of the result from \cite{sch91}:

\begin{theorem}\label{generator}
Let $\beta$ be a prehomogeneous dimension vector and 
$\underline{\beta}=(\beta_1,\cdots,\beta_t)$ be a
perpendicular sequence of the summands for the generic decomposition of $\beta$.
For each member $\gamma_i\in\underline{\beta}^{\perp}$
pick a generic $T_i\in R(Q,\gamma_i)$. Then the determinantal semi-invariants
$c_{\gamma_i}=c_{T_i},i=1,\cdots,n-t$ constitute an algebraically independent
system of generators for ${\bf k}[R(Q,\beta)]^{SL(\beta)}$.
\end{theorem}

For any real perpendicular sequence $\underline{\alpha}$ the sequence $\underline{\alpha}^{\perp}$
is a real quiver Schur sequence, so if $\underline{\alpha}$ is not a quiver Schur sequence,
then $^{\perp}(\underline{\alpha}^{\perp})$ is different from $\alpha$ (cf. Corollary \ref{preh_lss}).
Otherwise we have:

\begin{proposition}\label{vv}
If $\underline{\alpha}$ is a real quiver Schur sequence, then  
$\underline{\alpha}=^{\perp}(\underline{\alpha}^{\perp})$.
\end{proposition}

\begin{proof}
Denote  $\underline{\alpha}^{\perp}$ by $\underline{\beta}=(\beta_1,\cdots,\beta_{n-t})$
and $^{\perp}\underline{\beta}$ by $\underline{\gamma}=(\gamma_1,\cdots,\gamma_t)$.
Each $\alpha_i\in \underline{\alpha}$ has the property
$\alpha_i\perp \beta_j$ for  each $\beta_j\in \underline{\beta}$, hence, $\alpha_i$ decomposes 
as $\alpha_i = \rho_i^1\gamma_1+\cdots+\rho_i^t\gamma_t$ with non-negative integers $\rho_i^1,\cdots,\rho_i^t$.
By \cite[Proposition~13]{sh} the sequence $\underline{\rho}=(\rho_1,\cdots,\rho_t)$ is a real quiver
Schur sequence for the quiver $\Sigma_{\underline{\beta}}$. Since this quiver has $t$ vertices,
we have $^{\perp}\underline{\rho}$ is empty and therefore, there are no semi-invariants non-vanishing
on the representation $O$ corresponding to the decomposition $\rho_1+\cdots+\rho_t$.
However, by Theorem \ref{th1} $O$ is locally semi-simple, hence $O$ is the zero
point in $R(\Sigma_{\underline{\beta}},\dim O)$. In other words, $\rho_1,\cdots,\rho_t$
are the dimensions of simple representations, so $\underline{\alpha}=\underline{\gamma}$
up to order.
\end{proof}

Recall that the set of all locally semi-simple decompositions of
a dimension vector $\beta$ is the same as the Luna $GL(\beta)$-stratification of 
$R(Q,\beta)/\!\!/SL(\beta)$. For $\beta$ prehomogeneous we are able
to describe this stratification completely:

\begin{theorem}\label{prehstart}
Let $\beta$ be a prehomogeneous dimension vector and 
$\underline{\beta}=(\beta_1,\cdots,\beta_t)$ be a
perpendicular sequence of the summands for the generic decomposition of $\beta$.

{\bf 1.} There is a bijection of the Luna $GL(\beta)$-stratification of 
$R(Q,\beta)/\!\!/SL(\beta)$ with the set of subsequences in 
$\underline{\beta}^{\perp}=(\gamma_1,\cdots,\gamma_{n-t})$
$$
\underline{\gamma}\subseteq \underline{\beta}^{\perp} \rightarrow 
\beta=m_1\rho_1+\cdots+m_s\rho_s, (\rho_1,\cdots,\rho_s)=^{\perp}\underline{\gamma},
m_1,\cdots,m_s\in{\bf Z}_+.
$$

{\bf 2.} The stratum corresponding to $\underline{\gamma}$ is
$\{ \xi\in R(Q,\beta)/\!\!/SL(\beta)\vert c_{\gamma_i}(\xi)\neq 0\Leftrightarrow \gamma_i\in \underline{\gamma}\}$.

{\bf 3.} This bijection preserves the order on the sets: if 
$\underline{\gamma_1}\subseteq \underline{\gamma_2}$, then the stratum corresponding to 
$\underline{\gamma_1}$ is contained in the closure of that for $\underline{\gamma_2}$.
\end{theorem}

Before proving the Theorem we state an important

\begin{corollary}\label{preh_lss}
The generic locally semi-simple decomposition of $\beta$ is 
$\beta = m_1\beta_1'+\cdots+m_t\beta_t'$, where
$\underline{\beta'}=^{\perp}\underline{\beta}^{\perp}$.
\end{corollary}

\begin{proof}
{\bf 1.} First of all we show that the map is well-defined, i.e.,
there is a unique decomposition of $\beta$ in $^{\perp}\underline{\gamma}$
and this decomposition is locally semi-simple. By definition
of $\underline{\beta}^{\perp}$ we have for each $\beta_i\in \underline{\beta}$
and each $\gamma_j\in\underline{\gamma}$, $\beta_i\perp\gamma_j$.
Then by definition of $^{\perp}\underline{\gamma}$ $\beta_i$ is a ${\bf Z}_+$-linear 
combination of elements of $^{\perp}\underline{\gamma}$,
hence, $\beta$ is. This decomposition is locally semi-simple by Theorem \ref{char-lss}
and is unique, because $^{\perp}\underline{\gamma}$ is clearly linear independent.
To prove that our correspondance is a bijection we only need
to check that each locally semi-simple decomposition
$\beta=p_1\rho_1+\cdots+p_s\rho_s$ can be obtained this way.
First of all by Proposition \ref{prop3} we may assume that
$\underline{\rho}=(\rho_1,\cdots,\rho_s)$ is a real quiver Schur sequence.
In particular, there is a unique representation $V\subseteq R(Q,\beta)$ corresponding
to this decomposition, up to isomorphism. 
Then the sequence $\underline{\rho}^{\perp}$ is a subsequence
in $\underline{\beta}^{\perp}$. Indeed, each semi-invariant that
does not vanish on $V$ does not vanish generically on $R(Q,\beta)$
so each element of $\underline{\rho}^{\perp}$ is presented as a ${\bf Z}_+$-linear combination
over $\underline{\beta}^{\perp}$. Then $\underline{\rho}^{\perp}$
can be identified with the subsequence $\{i\in\{1,\cdots,n-t\}\vert c_{\gamma_i}(V)\neq 0\}$.
So we may set $\underline{\gamma}=\underline{\rho}^{\perp}$ and recover
$\underline{\rho}$ as $^{\perp}\underline{\gamma}$ by Proposition \ref{vv}.

{\bf 2.} The locally semi-simple representations over the stratum corresponding
to $\underline{\gamma}$ constitute one orbit $GL(\beta)V$, where
$V$ corresponds to the decomposition $\beta = m_1\rho_1+\cdots+m_s\rho_s$,
$\underline{\rho}=^{\perp}\underline{\gamma}$.
By definition of $^{\perp}\underline{\gamma}$,
$c_{\gamma_i}(V)\neq 0$ for $\gamma_i\in\underline{\gamma}$.
For any $j=1,\cdots,n-t$ it follows from the properties of determinatal semi-invariants that
$c_{\gamma_j}(V)\neq 0$ implies $\langle\rho_i,\gamma_j\rangle=0$ for each $i=1,\cdots,s$.
Since $\gamma_1,\cdots,\gamma_{n-t}$ are linear independent, the number
of semi-invariants $c_{\gamma_j}$ non-vanishing on $V$ is less than or equal to the codemension
of the common kernel of the corresponding forms $\langle\cdot,\gamma_j\rangle$
on ${\bf Q}^{Q_0}$. Since $\rho_1,\cdots,\rho_s$ are linear independent
and belong to that kernel, its codimension cannot be more than $n-s$, so
$c_{\gamma_j}(V)=0$ for $\gamma_j\notin\underline{\gamma}$. Assertion {\bf 3}
clearly follows from {\bf 2}.
\end{proof}

\begin{remark}
This assertion is closely related with \cite[Theorem~5.1]{dw6} 
because quiver Schur sequences are in bijection with the Luna 
$GL(\beta)$-strata by Proposition \ref{prop3} and the
subsequences in $\underline{\beta}^{\perp}$ are in bijection
with the faces of the cone ${\bf R}_+\Sigma(Q,\beta)$, which is simplicial
in this case.
\end{remark}
 

\section{Generic locally semi-simple decomposition}

We start with an obvious observation:
\begin{proposition}\label{prop4}
$R(Q,\alpha)^{GL(\alpha)}$  is generated by the scalar
endomorphisms of $V(a)$, $a\in Q_0$
corresponding to the loops $\varphi\in Q_1,t\varphi = h\varphi = a$.
\end{proposition}

Recall that the multiplicities of the summands 
in generic decompositions and in locally semi-simple
ones are of slightly different meaning. For the latter
the summand $m\beta$ means that the corresponding representation
has a direct summand $mS,\dim S= \beta$; on the other hand,
there can be summands as $m_1\beta + m_2\beta$ and this means
that the two indecomposable summands of the representation 
of dimension $\beta$ are non-isomorphic (hence, $\beta$ is imaginary).
In the generic decomposition the different summands are assumed to be distinct,
the multiplicity of $\beta$ with $q_Q(\beta) < 0$ is 1 by \cite{kac},
and for $q_Q(\beta)=0$ the direct summand $m\beta$ stands for
the sum of $m$ pairwise non isomorphic Schur representations of dimension $\beta$.
We now restrict the set of locally semi-simple decomposition we consider, as follows:

\begin{definition}
We call a locally semi-simple decomposition {\it almost loopless}
if the multiplicity of each imaginary Schur root $\beta$ is 1
and if $q_Q(\beta) < 0$, then this root occurs one time. 
\end{definition}
\begin{remark}\label{rem1}
Each locally semi-simple decomposition yields a loopless one with
the automorphism group of a smaller dimension. Namely,
the piece $m_1\beta_1+\cdots+m_t\beta_t$ of the decomposition 
with $\beta_1=\cdots =\beta_t=\beta$
such that $\beta$ is imaginary can be replaced either by the imaginary Schur root
$(m_1+\cdots+m_t)\beta$, if $q_Q(\beta) < 0$, or by the sum of $m_1+\cdots+m_t$ times
$\beta$, if $\beta$ is isotropic.
\end{remark}
\begin{remark}
If a locally semi-simple decomposition is almost loopless, then
we write it down in the style of generic one, with different summands being distinct.
\end{remark}

\begin{proposition}
Assume that $Q$ has no oriented cycles.
Let $\alpha = m_1\gamma_1+\cdots+m_t\gamma_t$ be the generic decomposition.
Then ${\bf k}[R(Q,\alpha)]^{SL(\alpha)}={\bf k}$ if and only if
$\vert Q_0\vert = t$.
\end{proposition}

\begin{proof}
Assume that ${\bf k}[R(Q,\alpha)]^{SL(\alpha)}\neq {\bf k}$ so that there
is a non-trivial determinantal semi-invariant $c_W,\dim W=\delta$.
Then by \cite{dw0} $c_W$ does not vanish on the summands of
generic representation, in particular, $\langle \gamma_i,\delta\rangle=0, i = 1, \cdots,t$.
On the other hand, by Remark 4.6 and Corollary 4.12 from \cite{dw6},
$\gamma_1,\cdots,\gamma_t$ are linear independent; since these are contained
in a proper ${\bf Q}$-vector subspace of  ${\bf Q}^{Q_0}$, we conclude $t<\vert Q_0\vert$.

Conversely, assume that ${\bf k}[R(Q,\alpha)]^{SL(\alpha)}= {\bf k}$.
This is equivalent to $SL(\alpha)$ acting with a dense orbit on $R(Q,\alpha)$
and, in particular $\alpha$ is a prehomogeneous dimension vector. 
But $(\gamma_1,\cdots,\gamma_t)^{\perp}$ must be empty because there are no non-trivial
semi-invariants. Then by Theorem \ref{th1.5}, $t=\vert Q_0\vert$.
\end{proof}

\begin{corollary}\label{cor1}
A locally semi-simple decomposition $\alpha=p_1\delta_1+\cdots+p_s\delta_s$
is generic if and only if
it is almost loopless, the local quiver $\Sigma_{\underline{\delta}}$ has no oriented
cycles except for the loops, and $s=t$.
\end{corollary}

\begin{proof}
By Remark \ref{rem1} the generic locally semi-simple
decomposition is almost loopless. 
By formula (\ref{slice}) the slice representation corresponding to the decomposition
is $(GL(\rho),R(\Sigma_{\underline{\delta}},\rho))$.
By \cite[Corollary~5]{sh} the
decomposition is generic locally semi-simple
if and only if 
${\bf k}[R(\Sigma_{\underline{\delta}},\rho)]^{SL(\rho)}$
is generated by the $GL(\rho)$-invariant submodule in $R(\Sigma_{\underline{\delta}},\rho)$.
The condition that $\underline{\delta}$ is almost loopless
is equivalent to the loops of $\Sigma_{\underline{\delta}}$ existing
only at vertices with dimension 1.
Removing the loops of $\Sigma_{\underline{\delta}}$ we define 
a quiver $\Sigma$ such that the generic and the generic
locally semi-simple decompositions of $\rho$ with respect to $\Sigma$
are the same as for $\Sigma_{\underline{\delta}}$.
By Proposition \ref{prop4} the above condition on 
$\Sigma_{\underline{\delta}}$ and $\rho$ is equivalent to
${\bf k}[R(\Sigma,\rho)]^{SL(\rho)}={\bf k}$.
If $\Sigma$ has loops, then the latter is false, in the opposite case
we apply the Proposition.
\end{proof}

In what follows we will intensively use the following well-known fact
\begin{proposition}
An exact sequence of homomorphisms $0\to U\to V\to W\to 0$
for representations of $Q$ yields for any representation $X$ exact sequences:
\begin{equation}\label{A,X}
0\to{\rm Hom}(W,X)\to{\rm Hom}(V,X)\to{\rm Hom}(U,X)\to\qquad\qquad\qquad\qquad\qquad
\end{equation}
$$\qquad\qquad\qquad\qquad\to{\rm Ext}(W,X)\to{\rm Ext}(V,X)\to{\rm Ext}(U,X)\to 0$$
\begin{equation}\label{X,A}
0\to{\rm Hom}(X,U)\to{\rm Hom}(X,V)\to{\rm Hom}(X,W)\to\qquad\qquad\qquad\qquad\qquad
\end{equation}
$$\qquad\qquad\qquad\qquad\to{\rm Ext}(X,U)\to{\rm Ext}(X,V)\to{\rm Ext}(X,W)\to 0$$
\end{proposition}

Our idea of algorithm for generic locally semi-simple decomposition
for a dimension vector $\alpha$ is similar to that for generic decomposition
from \cite{dw2}. That algorithm works with perpendicular
sequences and glue and permute two items each time there is a
non-trivial extension between them. We proceed as follows: starting from the generic decomposition,
we transform it into the generic locally semi-simple one
by slightly splitting summands by each other on the fact of non-trivial homomorphism. 
More precisely, we do it only when at least one of the items is imaginary, for the homomorphisms
between real summands we apply Corollary \ref{preh_lss}.
The most simple step of this procedure is: having Schur roots $\alpha,\beta$ with
$\alpha\perp\beta,ext(\beta,\alpha)=0$ but $hom(\beta,\alpha)\neq 0$,
we "factorize" the imaginary root by the real one to get
an imaginary Schur root with trivial homomorphism spaces
with the real root:

\begin{proposition}\label{permut}
Let $\alpha$ and $\beta$ be Schur roots such that $ext(\alpha,\beta)=0=ext(\beta,\alpha)$.

{\bf 1.} If both $\alpha$ and $\beta$ are imaginary, then $hom(\alpha,\beta)=0=hom(\beta,\alpha)$.

{\bf 2.} In any case either $hom(\alpha,\beta)=0$ or $hom(\beta,\alpha)=0$.

{\bf 3.} Assume that $hom(\alpha,\beta)=0,\dim hom(\beta,\alpha)=p>0$.

\hspace{0.5cm} {\bf A}: If $\alpha$ is imaginary, then for generic $A\in R(Q,\alpha),B\in R(Q,\beta)$ there is an exact sequence of homomorphisms:
$0\to pB\to A\to C\to 0$, $C$ is Schurian, $q_Q(\dim C)=q_Q(\alpha)$,
${\rm Hom}(C,B)=0,{\rm Hom}(B,C)=0,{\rm Ext}(B,C)=0$.

\hspace{0.5cm} {\bf B}: If $\beta$ is imaginary, then for generic $A\in R(Q,\alpha),B\in R(Q,\beta)$ there is an exact sequence of homomorphisms:
$0\to C\to B\to pA\to 0$, $C$ is Schurian, $q_Q(\dim C)=q_Q(\beta)$,
${\rm Hom}(A,C)=0,{\rm Hom}(C,A)=0,{\rm Ext}(C,A)=0$.

Moreover, in both cases {\bf A,B}, 
if $\gamma\perp\alpha$ and $\gamma\perp\beta$, then $\gamma\perp\dim C$;
if $\alpha\perp\gamma$ and $\beta\perp\gamma$, then $\dim C\perp\gamma$.
\end{proposition}

\begin{proof}
Assertions {\bf 1} and {\bf 2} follow from Theorems 4.1 and 2.4 of \cite{sch92}, respectively.
We prove {\bf 3A} the proof for {\bf 3B} being similar.
Consider the decomposition $\rho=\alpha+\beta$; the conditions $ext(\alpha,\beta)=0=ext(\beta,\alpha)$
imply that this is a generic one.
Then by  Corollary \ref{cor1} the generic locally semi-simple decomposition
of $\rho$ has two summands, $\rho = m_1\gamma+m_2\delta$. If both summands $\gamma,\delta$
would be imaginary, then by {\bf 1} that decomposition would be generic and different from 
$\rho=\alpha+\beta$. We claim that at least one of  $\gamma,\delta$ is imaginary.
Assume not, then the local quiver $\Sigma$ of that decomposition has no loops, so
it only has $r$ arrows of the same orientation.
By \cite[Proposition~14]{sh} the generic decomposition for $(\Sigma,(m_1,m_2))$
would give $\rho=\alpha+\beta$ when dimension vectors for $\Sigma$ are converted in those for $Q$.
So the generic decomposition for $(\Sigma,(m_1,m_2))$ has two summands, one imaginary and one
real. However, it is shown in \cite{dw2} for this type of quivers that either a dimension vector
is an imaginary Schur root, or it has two real summands in the generic decomposition. 
This contradiction proves our claim.

So the generic locally semi-simple decomposition of $\rho$ has the form $\rho = a\gamma+\delta$, 
where $\delta$ is an imaginary Schur root and $\gamma$ is real. Again, the local quiver
$\Sigma$ of that decomposition has $r$ arrows of the same orientation
between two vertices and besides $1-q_Q(\delta)$ loops at the vertex corresponding to $\delta$.
Consider the generic decomposition of $(a,1)$ on $\Sigma$, it must have a form
$(a,1)=(a-q)(1,0)+(q,1), q > 0$. The first necessary condition on $q$ 
is that $(q,1)$ should be a root, hence $q\leq r$.
Further, by \cite{kac} the Euler form is non-negative on the
pairs of summands of the generic decomposition;
$\langle (q,1),(1,0)\rangle_{\Sigma}$ and $\langle (1,0),(q,1)\rangle_{\Sigma}$
are equal $q$ and $q-r$ up to transposition, so $q=r$.
On the other hand, by Propositions 14, 15 from \cite{sh} the pair 
of values of the Euler form for these summands must be
same as for $\alpha$ and $\beta$ on $Q$, so $r=p$. 
We therefore proved that the generic locally semi-simple decomposition of $\rho$ is 
$\rho=(p+1)\beta+ \alpha -p\beta$. In particular, $\alpha -p\beta$
is a Schur root and $\beta\perp \alpha -p\beta$. Hence, by \cite[Theorem~3.3]{sch92}
generic representation of dimension $\alpha$ has a subrepresentation
of dimension $p\beta$, so isomorphic to $pB$ because generic.
Thus we have the claimed exact sequence.
Write: $q_Q(\alpha-p\beta)=\langle \alpha-p\beta,\alpha-p\beta\rangle=
q_Q(\alpha)+p^2 - p(\langle\alpha,\beta\rangle+\langle\beta,\alpha\rangle)=q_Q(\alpha)$.

The rest of the assertion will be deduced from formulae (\ref{A,X},\ref{X,A})
for $U=pB,V=A,W=C$. Since ${\rm Ext}(V,U)=0$ and ${\rm Hom}(V,U)=0$,
(\ref{X,A}) with $X=V$ yields ${\rm Hom}(V,W)\cong{\rm Hom}(V,V)={\bf k}$.
Then (\ref{A,X}) with $X=W$ yields ${\rm Hom}(W,W)={\bf k}$, so $C$ is Schurian.
Next, ${\rm Hom}(V,U)=0$ and  (\ref{A,X}) with $X=U$ yield ${\rm Hom}(W,U)=0$,
hence, ${\rm Hom}(C,B)=0$.
Finally, apply (\ref{X,A}) with $X=U$ and note that ${\rm Ext}(U,V)=0$
by assumption, ${\rm Ext}(U,U)=0$ because $\beta$ is real,
${\rm Hom}(U,U)\cong {\rm Hom}(U,V)\cong {\bf k}^{p^2}$;
then ${\rm Ext}(U,W)$ and ${\rm Hom}(U,W)$ vanish.
That $\dim C$ is perpendicular to any dimension vector $\gamma$, which
is perpendicular to $\alpha$ and $\beta$ follows directly from
formulae (\ref{A,X},\ref{X,A}) with $X$ being generic representation
of dimension $\gamma$; with such a choice of $X$ four 
members in the exact sequence vanish, hence all six vanish. 
\end{proof}

\begin{definition}
Assume that $\alpha,\beta$  are the subsequent members of a decomposition.
If $\alpha,\beta$ meet the conditions of \ref{permut}.3A (resp. \ref{permut}.3B),
we call the replacement of $\alpha,\beta$
by $\beta,\alpha-p\beta$ (resp. $\beta-p\alpha,\alpha$)
{\it pushing $\alpha$ right} (resp. {\it pushing $\beta$ left}).
This operation also includes the obvious recaluclation
of multiplicities on the decomposition. We also may apply both
terms to the transposition of the members $\alpha,\beta$ (even when both are
real) such that $\alpha\perp\beta$ and $\beta\perp\alpha$.
\end{definition}

Now we present our algorithm for the generic locally semi-simple
decomposition of a dimension vector $\alpha$ with a given generic
decomposition $\alpha = m_1\gamma_1+\cdots+m_t\gamma_t$ such that
$\underline{\gamma}$ is a Schur sequence. Recall again that
the decomposition of this sort is the result of the algorithm from \cite{dw2}.

\begin{algorithm}\label{alg_lss}\quad

{\sc input:} a quiver $Q$ with $n$ vertices and without oriented cycles,

\hspace{1.2cm}a quiver sequence $(\gamma_1,\cdots,\gamma_t)$
of the summands  

\hspace{1.2cm}for the generic decomposition, multiplicities $(m_1,\cdots,m_t)$

{\sc output:} a quiver sequence  $(\alpha_1,\cdots,\alpha_t)$ of the 
summands 

\hspace{1.5cm}for the generic locally semi-simple decomposition, 

\hspace{1.5cm}multiplicities $(p_1,\cdots,p_t)$

  $>>$ First stage: 

  {\sc for} $i=2,\cdots,t$
  
\hspace{0.5cm} Set $j=i$

\hspace{0.5cm} {\sc while} $\gamma_j$ is imaginary and $\gamma_{j-1}$ is real 

\hspace{1.5cm}  push $\gamma_j$ left and $j=j-1$

  $>>$ Result of the first stage: first $s \leq t$ members of $\underline{\gamma}$ are imaginary,
  
  $>>$ last $t-s$ are real.

  $>>$ Second stage: 

  Replace the subsequence $\underline{\gamma}'=(\gamma_{s+1},\cdots,\gamma_t)$ by 
  $^{\perp}\underline{\gamma}'^{\perp}$

  $>>$ Result of the second stage: homorphisms for new $(\gamma_{s+1},\cdots,\gamma_t)$
  are trivial

  $>>$Third stage:

   {\sc while} there is $1\leq i < j\leq t$ with $\langle \gamma_j,\gamma_i\rangle > 0$

\hspace{0.5cm} gurantee that the segment $[i,j]$ is minimal with such a property

\hspace{0.5cm} transfer $\gamma_j$ to position $i+1$

\hspace{0.5cm}  push $\gamma_i$ right

$>>$ Result of the third stage: the sequence $\underline{\gamma}$ is now a quiver Schur sequence.

{\sc return} the current $\gamma_1,\cdots,\gamma_t$ and $(m_1,\cdots,m_t)$.
\end{algorithm}

Now we are going to prove the algorithm.

\begin{proposition}
While the first stage of the algorithm the sequence  $\underline{\gamma}$
remains perpendicular and each time when $\gamma_j$ is imaginary and $\gamma_{j-1}$ is real
$ext(\gamma_j,\gamma_{j-1})=0$.
\end{proposition}

\begin{proof} By Proposition \ref{permut} the sequence $\underline{\gamma}$ remains
orthogonal after pushing $\gamma_j$ left provided $\underline{\gamma}$
was perpendicular before it. We need to prove additionaly that $C$ in
the exact sequence from Proposition \ref{permut}.3B has the property
${\rm Ext}(C,D)=0$ for $D$ being generic representation
of dimension $\gamma_q, q < j-1$. This follows from (\ref{A,X})
and the fact that ${\rm Ext}(B,D)=0$ by induction.
\end{proof}

\begin{proposition}
After the first stage $\gamma_i\perp\gamma_j$ for $1\leq i\neq j\leq s$.
\end{proposition}
\begin{proof}
Assume that $i < j$. Then  $\gamma_i\perp\gamma_j$ follows from the previous
Proposition and $hom(\gamma_j,\gamma_i)=0$ follows
from \cite[Theorem~4.1]{sch92}, because both $\gamma_i$ and $\gamma_j$
are imaginary. Assume that $ext(\gamma_j,\gamma_i)>0$.
Then we can apply the Algorithm for generic decomposition from \cite{dw2}
to the perpendicular sequence $\underline{\gamma}$.
From that Algorithm follows that in such a situation
we can glue together  $\gamma_i$ and $\gamma_j$
and get a  perpendicular sequence with $t-1$ members
and, continuing the Algorithm, we get the generic decomposition
with less than $t$ members. Contradiction.
\end{proof}

\begin{proposition}
After the second stage of the algorithm the sequence $\underline{\gamma}$ remains
orthogonal and $hom(\gamma_i,\gamma_j) > 0$ implies $i > s, j \leq s$.
\end{proposition}

\begin{proof}
Before the second stage $\underline{\gamma}$ consists
of two segments, $(\gamma_1,\cdots,\gamma_s)$ and $\underline{\gamma}'=(\gamma_{s+1},\cdots,\gamma_t)$.
Both segments are orthogonal sequences with trivial ${\rm Ext}$ spaces,
the first segment by the previous Proposition and $\underline{\gamma}'$ because
it is a subsequence in the original $\underline{\gamma}$. Consider
two dimension vectors $\rho_1,\rho_2$ being the linear combinations
of the two segments with the multiplicities.
The fact that $\underline{\gamma}$ is orthogonal
is equivalent to $\rho_1\perp\rho_2$, that is,
generic representation $R_1$ of dimension $\rho_1$ is perpendicular to that in
dimension $\rho_2$. The second stage  consists in replacing
the generic decomposition for $\rho_2$ by the generic locally semi-simple one.
Therefore $\underline{\gamma}$ remains orthogonal because $R_1$ is peprpendicular to 
generic locally semi-simple representation of dimension $\rho_2$ otherwise
the determinantal semi-invariant defined by $R_1$ vanishes on $R(Q,\rho_2)$.
The fact about ${\rm Hom}$-spaces follows from the previous Proposition and
the feature of generic locally semi-simple decompositions.
\end{proof}

The third stage of the algorithm is based on the following
\begin{lemma}\label{perekid}
If $\alpha,\beta,\gamma$ is a perpendicular sequence of Schur roots
such that $\alpha$ is imaginary and $hom(\gamma,\alpha)>0$, then $ext(\gamma,\beta)=0$. 
\end{lemma}
\begin{proof}
Assume $ext(\gamma,\beta)>0$. 
Pick generic representations $U\in R(Q,\alpha),V\in R(Q,\beta),W\in R(Q,\gamma)$ and
consider a non-split exact sequence $0\to V\to X \xrightarrow{p}W\to 0$.
Then $X$ is Schurian (see the proof of \cite[Corollary~12]{dw6}) 
and $U\perp V, U\perp W$ together with (\ref{X,A}) 
imply $U\perp X$. Pick a non-trivial homomorphism $h\in{\rm Hom}(W,U)$.
By \cite[Lemma~2.3]{sch92} $h$ is either injective or surjective.
Since $\alpha$ is imaginary, by Proposition \ref{permut}.3A $h$ must be
injective. Then the composition $hp\in {\rm Hom}(X,U)$ is not surjective.
But the kernel of $hp$ contains the image of $V$ so $hp$ is neither
injective nor surjective in contradiction with \cite[Lemma~2.3]{sch92}.
\end{proof}

\begin{proposition} 
While the third stage of the algorithm assume $\langle \gamma_j,\gamma_i\rangle > 0$

{\bf 1.} if $[i,j]$ is minimal with the property and $\gamma_i$ is imaginary, 
          then we have $\langle \gamma_j,\gamma_k\rangle = 0$ for $i<k<j$.

{\bf 2.}  $\gamma_j$ is real and $\gamma_i$ is imaginary.

{\bf 3.} The third stage of the algorithm finishes after finitely many steps. 
\end{proposition}

\begin{proof}
First of all, for any $i < j$ holds $ext(\gamma_i,\gamma_j)=0$ and by \cite[Theorem~4.1]{sch92}
either  $ext(\gamma_j,\gamma_i)=0$ or  $hom(\gamma_j,\gamma_i)=0$,
and the latter is the case if both $\gamma_i$ and $\gamma_j$ are imaginary.
So either $\langle \gamma_j,\gamma_i\rangle < 0$ and in this case
$ext(\gamma_j,\gamma_i)>0$ and $hom(\gamma_j,\gamma_i)=0$,
or  $\langle \gamma_j,\gamma_i\rangle > 0$ and in this case
$ext(\gamma_j,\gamma_i)=0$ and $hom(\gamma_j,\gamma_i)>0$, or else
$\langle \gamma_j,\gamma_i\rangle = 0$ and in this case
$ext(\gamma_j,\gamma_i)=0=hom(\gamma_j,\gamma_i)$.

{\bf 1.}  Applying Lemma  \ref{perekid} to the sequence $\gamma_i,\gamma_k,\gamma_j,i<k<j$,
we get $ext(\gamma_j,\gamma_k)$$=0$. So we have $\langle \gamma_j,\gamma_k\rangle\geq 0$ and
$\langle \gamma_j,\gamma_k\rangle >0$ contradicts to the minimality of $[i,j]$.

{\bf 2.} In the third stage we do not change the real roots, only permute them,
hence, exactly one of $\gamma_i$ and $\gamma_j$ is imaginary. 
We prove that $\gamma_j$ is real and $\gamma_i$ is imaginary
applying induction.
At the beginning each imaginary is to the left of each real.
Then we apply the step of the third type to a minimal
segment $[i',j']$ and have $\gamma_{i'}$ imaginary by induction.
Then by {\bf 1} $\gamma_{j'}$ is perpendicular from both sides 
to $\gamma_k,i'<k<j'$ so the property remains true
for all pairs containing $\gamma_{j'}$. As for the
imaginary root being the result of pushing $\gamma_{i'}$ right,
it is given by Proposition \ref{permut}.3A for $\alpha = \gamma_{i'},\beta = \gamma_{j'}$
and we have the exact sequence $0\to pB\to A\to C\to 0$ from
\ref{permut}.3A. Taking $X$ to be a generic representation of dimension $\gamma_l,l<i'$,
we have ${\rm Hom}(A,X)=0$, because $\langle\gamma_{i'},\gamma_{l}\rangle\leq 0$ by induction.
Then by  (\ref{A,X}), ${\rm Hom}(C,X)=0$, so $hom(\gamma_{i'}-p\gamma_{j'},\gamma_{l})=0$,
hence, $\langle\gamma_{i'}-p\gamma_{j'},\gamma_{l}\rangle\leq 0$ and the property remains true.

{\bf 3.} Throughout the stage we decrease the imaginary members of $\underline{\gamma}$
so we can not do it infinitely many times. 
\end{proof}

\begin{theorem}
Algorithm \ref{alg_lss} yields the generic locally semi-simple decomposition.
\end{theorem}

\begin{proof}
The above Propositions convinced us that after finitely many
steps we obtain a decomposition with $\gamma$ being a perpendicular
sequence of Schur roots such that $hom(\gamma_i,\gamma_j)=0$ if $i\neq j$.
We also claim that this is a Schur sequence, that is, $\gamma_i\circ\gamma_j=1$
for $i < j$. Indeed this was true for the starting sequence
and by \cite[Lemma~4.2]{dw6} this needs to be checked only 
for $\gamma_i$ and $\gamma_j$ being imaginary. We now show that
the property is preserved by any pushing of the imaginary root.
So assume that $\gamma_i$ and $\gamma_j$ are imaginary
$\gamma_k$ is real, $i < k$ and we replace $\gamma_j$ by $\gamma_j-p\gamma_k$.
Then the vector space of the semi-invariants on $R(Q,\gamma_i)$
with the weight $-\langle\cdot,\gamma_j-p\gamma_k\rangle$
is embedded to that of the weight $-\langle\cdot,\gamma_j\rangle$
by multiplying with a semi-invariant of the weight $-\langle\cdot,p\gamma_k\rangle$.
So $\gamma_i\circ\gamma_j-p\gamma_k>1$ would imply 
$\gamma_i\circ\gamma_j>1$. A similar argument works for pushing $\gamma_i$
so we proved that the output of the algorithm is a quiver Schur sequence.
Hence, by Theorem \ref{th1} the
resulting decomposition is locally semi-simple and by Corollary \ref{cor1}
this is the generic locally semi-simple decomposition.
\end{proof}

\end{document}